\documentclass[a4paper,11pt]{amsart}

\usepackage{amsmath}
\usepackage{amssymb}
\usepackage{amsthm}
\usepackage[all,cmtip]{xy}
\usepackage{textcomp}
\usepackage{slashed}
\usepackage{enumerate}
\usepackage{lipsum}
\usepackage{caption}

\newtheorem{thm}{Theorem}[section]
\newtheorem{lemma}[thm]{Lemma}
\newtheorem{df}{Definition}[section]
\newtheorem{remark}{Remark}[section]
\newtheorem{ex}{Example}[section]

\numberwithin{equation}{section}
\newtheorem{cor}{Corollary}[section]

\newcommand{\Sp}{\mathrm{Sp}}
\newcommand{\GL}{\mathrm{GL}}

\newcommand{\Fl}{F_{1,2}}

\newcommand{\R}{\mathbb R}

\newcommand{\C}{\mathbb C}

\newcommand{\Q}{\mathbb H}
\newcommand{\Z}{\mathbb Z}

\newcommand{\QP}{\mathbb{HP}}
\newcommand{\CP}{\mathbb{CP}}

\newcommand{\cE}{\mathcal E}
\newcommand{\cO}{\mathcal O} 

\newcommand{\PT}{\mathcal P}
\newcommand{\cC}{\mathcal C}

\newcommand{\cgr}{V_2(\C^{2n+2})} 

\newcommand{\bpar}{\bar\partial} 
\newcommand{\ra}{\rightarrow}
\newcommand{\wt}{\widetilde}

\newcommand{\mh}{\mathcal{H}}

\newcommand{\coker}{\mathrm{coker}}

\newcommand{\im}{\mathrm{im}}

\title{Monogenic hull for the $n$-Cauchy-Fueter operator and twistor theory}
\author{Tom\'a\v s Sala\v c}
\thanks{
2010
Mathematics  Subject  Classification.
Primary: 35N05, 32L25.\\
Key words and phrases. Twistor theory, Penrose transform, $n$-Cauchy-Fueter operator, monogenic hull.\\
\\}
\begin{document}

\begin{abstract}
This is the first part in a series of three articles in which are studied the domains of monogenicity for the $n$-Cauchy-Fueter operator. Using the twistor theory, we will in this article show that for a given  open subset  $U$ of $\Q^n$, there is an open subset $\mh(U)$, called the monogenic hull of $U$, of $M_{2n\times2}^\C=\Q^n\otimes \C$ such that each monogenic function in $U$ extends to a unique pair of holomorphic functions on $\mh(U)$.

In the second part of the series we will exploit the twistor theory furthermore to prove that each pseudoconvex domain in $\Q^n$ is a domain of monogenicity. In the third part of the series, we show  the other implication and provide a geometric characterization of the domains of monogenicity.
\end{abstract}

\maketitle

\section{Introduction}\label{section introduction}

We will write  $\Q^n=\{(q_1,\dots,q_n):\ q_\ell=x_0^\ell+ix_1^\ell+jx_2^\ell+kx_3^\ell\in\Q,\ \ell=1,\dots,n\}$.
Let $U$ be an open subset of $\Q^n$  and $\psi:U\ra\Q$ be smooth. We put
\begin{equation}\label{dbar operator}
 \partial_{\bar q_\ell} \psi:=\frac{\partial \psi}{\partial x^\ell_0}+i\frac{\partial \psi}{\partial x_1^\ell}+j\frac{\partial \psi}{\partial x^\ell_2}+k\frac{\partial \psi}{\partial x^\ell_3}, \ \ell=1,\dots,n
\end{equation}
and call
\begin{equation}\label{n-CF op}
D\psi:=(\partial_{\bar q_1}\psi,\dots,\partial_{\bar q_n}\psi)
\end{equation}
the $n$-\textit{Cauchy-Fueter operator}. If $D\psi=0$, then $\psi$ is called \textit{monogenic} (or \textit{regular}) and  we denote by $\mathcal{R}(U)$ the space of monogenic functions  in $U$. See \cite{BAS}, \cite{BS}, \cite{CSSS} and \cite{CSS} for some background on this operator.

Fixing a linear isomorphism $\Q\ra\C^2$ as in Section \ref{section background on Q}, the function $\psi$ corresponds to a pair of functions $\psi_{A'}:U\ra\C,\ A=0,1$. If $\psi\in\mathcal{R}(U)$, then it is well known that $\psi_{0'},\psi_{1'}$ are analytic and since $\Q^n\otimes_\R\C$ is isomorphic to the space $M_{2n\times2}^\C$ of complex $2n\times2$ matrices, there is an open set $U_\C\subset M_{2n\times2}^\C$ and unique holomorphic functions $\psi^\C_{A'}:U_\C\ra\C$ such that $U_\C\cap\Q^n=U$  and $\psi^\C_{A'}|_{U}=\psi_{A'},\ A=0,1$. The main result of this paper is (see Theorem \ref{main thm}) that there is an open subset $\mh(U)$ of $M_{2n\times2}^\C$, called the \textit{monogenic hull} of $U$, with  $\mh(U)\cap\Q^n=U$ such that each monogenic function in $U$ extends to a  pair of holomorphic functions in $\mh(U)$. 

If $n=1$, then it is easy to see that $\mh(U)$ is maximal among all open subsets of $M_{2\times2}^\C$ which satisfy the extension requirement (see Example \ref{example fundamental solution in mh Qast}). If $n>1$, then the situation is more subtle. This is related to the fact that $D$ is an overdetermined operator and  Hartog's phenomenon holds for monogenic functions (this is originally due to \cite{Pe}, see also \cite[Theorem 3.3.5]{CSSS}, \cite{Wa} and \cite{WaI}). Hence, the theory of monogenic functions of several quaternionic variables is parallel to the theory of holomorphic functions.
Actually, we will use in the third part of the series \cite{TSI} the main result of this article to show that any domain of monogenicity\footnote{Loosely speaking, we call the open set $U$ a domain of monogenicity if there is no open subset $V$ of $\Q^n$ with $U\subsetneq V$ such that the restriction map $\mathcal{R}(V)\ra\mathcal{R}(U)$ is surjective. See \cite{TS} or \cite{TSI} for a precise definition.} is pseudoconvex\footnote{We call  $U$ pseudoconvex if it admits a smooth exhausting $\Q$-plurisubharmonic function or equivalently, $\delta^{-2}$ is $\Q$-plurisubharmonic where $\delta$ is the usual distance function to the boundary of $U$. See \cite{TSI}}.

In order to prove the main result, we will use the twistor theory as in \cite{EX}, see also \cite{B}, \cite{BE}, \cite{E} \cite{WaI} and \cite{WW}. Recall \cite[Section 4.4.9]{CS} that there is a fiber bundle $S^2\ra\wt U\xrightarrow{\tau} U$, called the \textit{twistor space}, associated to the flat almost quaternionic structure over $U$. The total space $\wt U$ carries a tautological almost complex structure which is integrable and thus, $\wt U$ is a complex manifold. In this article we will view $\wt U$ as an open submanifold of $\CP^{2n+1}$. The hardest part of the proof of Theorem \ref{main thm} is (see Theorem \ref{thm isom cohomology and monogenic functions}) to show  that there is an isomorphism 
\begin{equation}\label{isomorphism of cohomology groups and monogenic kernel over U}
 H^1(\wt U,L)\ra\mathcal{R}(U)
\end{equation}
where $L$ is a certain holomorphic line bundle over $\wt U$. This extends results given in (\cite{KW}). The isomorphism (\ref{isomorphism of cohomology groups and monogenic kernel over U}) is given by some completely explicit integral formula and is coming from the Penrose transform. 

In the second part \cite{TS} of the series, we will exploit the twistor theory furthermore. Using $L^2$ estimates as in \cite{H}, we will show that $H^2(\wt U,L)=0$ when $U$ is pseudoconvex and from this information  we will conclude that $U$ is a domain of monogenicity. 

\section*{Notation}

\begin{itemize}
\item $M_{n\times k}^T$= matrices of size $n\times k$ with coefficients in a field $T$
\item $T^\ast=T\setminus\{0\}$
\item $\Sp(1)=$ the group of unit quaternions
\end{itemize}

\section*{Acknowledgment}
The author is grateful to Vladim\'ir Sou\v cek for  stimulating discussions and to Roman L\'avi\v cka for his help.
The author gratefully acknowledges the support of the grant 17-01171S of the Grant Agency of the Czech Republic.

\section{$n$-Cauchy-Fueter operator}

\subsection{Some background on quaternions}\label{section background on Q}

Let $\Re\Q$ and $\Im\Q$ be the subspace of real and imaginary quaternions, respectively. We denote by $(x,y):=\Re(\bar xy),\ x,y\in\Q^n$ the standard real inner product  and  by $\|x\|:=\sqrt{(x,x)}$ the associated norm. We will always view $\Q^n$ as a \textbf{right} $\Q$-vector space and  $\C$ as the subalgebra of $\Q$ generated by $i$. Then $\Q^n$ is  a complex vector space of dimension $2n$ and the map 
\begin{equation}\label{lin isom Qn and C2n}
\C^{2n}\ra\Q^n,\ \ \
(\alpha_1,
\beta_2,\dots,\alpha_n,\beta_n)^T
\mapsto (\alpha_1+k\beta_1,\dots,\alpha_n+k\beta_n)^T.
\end{equation}
is a $\C$-linear isomorphism. Using the notation set in Introduction,  we may write
\begin{equation}\label{complex coordinates on Q}
\alpha_\ell=x_0^\ell+ix_1^\ell, \ \beta_\ell=x_3^\ell+ix_2^\ell\ \ \mathrm{and}\ \   \partial_{\bar q_\ell}=2(\partial_{\bar\alpha_\ell}+k\partial_{\bar\beta_\ell}),\ \ell=1,\dots,n.
 \end{equation}
The map $\Q^n\ra\Q^n,\ w\mapsto wk$ corresponds to
\begin{equation}
\mathbb K:\C^{2n}\ra\C^{2n},\ \mathbb  K(\alpha_1,\beta_1,\dots,\alpha_n,\beta_n)^T=(-\bar\beta_1,\bar\alpha_1,\dots,-\bar\beta_n,\bar\alpha_n)^T.
\end{equation}

Let $T=\C$ or $\Q$. As any $\Q$-linear map $\Q^k\ra\Q^n$ is also complex linear, there is an injective homomorphism of algebras $M_{n\times k}^\Q\hookrightarrow M_{2n\times2k}^\C$. On the other hand, a complex linear map $A:\C^{2k}\ra\C^{2n}$ is $\Q$-linear if and only if $A\circ\mathbb  K=\mathbb K\circ A$.
If $n=k=1$, then the embedding is
\begin{equation}\label{embedding Q into M22}
\C^2=\Q= M_{1\times 1}^\Q\hookrightarrow M_{2\times 2}^\C,\
\left(
\begin{matrix}
\alpha\\
\beta\\ 
\end{matrix}
\right)
\mapsto\left(
\begin{matrix}
\alpha&-\bar\beta\\
\beta&\bar\alpha\\
\end{matrix}
\right)
\end{equation}
and more generally:
\begin{equation}\label{embedding Qn into M2n2}
\C^{2n}=\Q^n= M(n,1,\Q)\hookrightarrow M_{2n\times2}^\C,\
x\mapsto \mathrm{M}(x):=(x|\mathbb K(x)).
\end{equation}
We will use the isomorphisms in (\ref{embedding Qn into M2n2}) without further comment.
Given  $(z_{AA'})\in M_{2\times 2}^\C$, there are unique  $\alpha,\beta,\gamma,\delta\in\C$ such that 
\begin{equation}\label{decomposition of C matrices}
 \left(
\begin{matrix}
z_{00'}&z_{01'}\\
z_{10'}&z_{11'}\\
\end{matrix}
\right)=\left(
\begin{matrix}
\alpha&-\bar\beta\\
\beta&\bar\alpha\\
\end{matrix}
\right)+
i\left(
\begin{matrix}
\gamma&-\bar\delta\\
\delta&\bar\gamma\\
\end{matrix}
\right).
\end{equation}
This shows  that $M_{2n\times2k}^\C=M_{n\times k}^\Q\otimes_\R\C$ and thus, we can view a matrix $z\in M_{2n\times2k}^\C$ as a pair $(x,y)$ where $x,y\in\Q^n$. We will do that without further comment. We will work with the norm $\|(x,y)\|_\C:=\sqrt{\|x\|^2+\|y\|^2}$ on $M_{2n\times2}^\C$.

\subsection{$n$-Cauchy-Fueter operator and the monogenic hull}
Given $\psi:U\ra\Q$, there are unique functions $\psi_{A'}:U\ra\C,\ A=0,1$ so that $\psi=\psi_{0'}+k\psi_{1'}$. Using \ref{complex coordinates on Q}, we see that $D\psi=0$ if and only if   
\begin{align}\label{CF I}
&\partial_{\beta_\ell}\psi_{1'}-\partial_{\bar\alpha_\ell}\psi_{0'}=0,\\
&\partial_{\alpha_\ell}\psi_{1'}+\partial_{\bar\beta_\ell}\psi_{0'}=0\nonumber
\end{align}
for every  $\ell=1,\dots,n$.
As $M_{2n\times2}^\C$ is the complexification of the totally real submanifold $\Q^n$ and any monogenic function is real analytic, it follows that there is an open subset $U_\C$ of $M_{2n\times2}^\C$ with a pair of holomorphic functions $\psi^\C_i:U_\C\ra\C$ such that:
\begin{enumerate}[(i)]
 \item $U_\C\cap\Q^n=U$ and
 \item $\psi^\C_{A'}|_{U}=\psi_{A'},\ A=0,1$.
\end{enumerate}
Moreover, it is well known that $(\psi^\C_{A'})_{A=0,1}$ are null solutions of 
\begin{equation}\label{hol n-CF op}
D^\C(\psi_{0'}^\C,\psi_{1'}^\C)= (\partial_{z_{A0'}}\psi^\C_{1'}-\partial_{z_{A1'}}\psi^\C_{0'})_{A=0,1,\dots,2n-1}.
\end{equation}
Conversely, if $(\psi_{A'}^\C)_{A=0,1}$ are holomorphic in $U_\C$ and are null solutions of (\ref{hol n-CF op}), then the restriction of these functions to $U$ is monogenic.
If $V$ is an open subset of $M_{2n\times2}^\C$, then we put $$\mathcal{R}^\C(V):=\{(\psi_{A'}^\C)_{A=0,1}:\ D^\C(\psi_{0'}^\C,\psi_{1'}^\C)=0\}.$$

\begin{df}\label{df mh}
Let $U\subset\Q^n$ be open. We call the set 
\begin{equation}\label{condition on monogenic hull}
\{(x,y)\in M_{2n\times2}^\C|\ \forall  q\in \Sp(1)\cap\Im\Q:\ x+yq\in U\}
\end{equation}
the monogenic hull  $\mh(U)$ of $U$.
\end{df}

It follows from the definition  that $\mh(U)$ is  open,  $U=\mh(U)\cap \Q^n$ and that $\mh(U)=\bigcap_{x\in U^\complement}\mh(\Q^n\setminus\{x\})$. Consider also the following example.

\begin{ex}\label{example mh}
Let $U$ be an open subset of $\Q$. Since $\det(x,y):=\|x\|-\|y\|+2i(x,y)$, it  is  clear  that $\mh(\Q^\ast)=\GL(2,\C)$. As $\GL(2,\C)$ is the complement of the analytic variety $\{z_{00'}z_{11'}-z_{01'}z_{10'}=0\}$ in $M_{2\times2}^\C$, it is a domain of holomorphy. By \cite[Theorem 2.6.9]{H}, also $\mh(U)$ is a domain of holomorphy.
\end{ex}

\begin{remark}
As we have seen in Example \ref{example mh}, $\mh(U)$ is a domain of holomorphy for any open subset $U$ of $\Q$. We will show in the third part of the series that $\mh(U)$ is a domain of holomorphy if and only if $U$ is a domain of monogenicity in $\Q^n$.
\end{remark}

The main result of this article is the following Theorem.

\begin{thm}\label{main thm}
Let $U$ be an open subset of $\Q^n$. Then the restriction map 
\begin{equation}\label{restriction map}
\mathcal{R}^\C(\mh(U))\ra\mathcal{R}(U),\ (\psi_{A'}^\C)_{A=0,1}\mapsto(\psi_{A'}^\C|_U)_{A=0,1}
\end{equation}
is an isomorphism. 
\end{thm}
The proof of Theorem \ref{main thm} will occupy the rest of this article. Let us consider the following example.

\begin{ex}\label{example fundamental solution in mh Qast}
The fundamental solution of the Cauchy-Fueter operator is (up to constant)
\begin{equation*}
 E(q)=\frac{\bar q}{|q|^4}=\frac{\bar\alpha-k\beta}{(\alpha\bar\alpha+\beta\bar\beta)^2}
\end{equation*}
which is  monogenic in $\Q^\ast$. Hence, the corresponding complex functions are
\begin{equation*}
 \psi_{0'}=\frac{\bar\alpha}{(\alpha\bar\alpha+\beta\bar\beta)^2}\ \mathrm{and}\ \psi_{1'}=\frac{-\beta}{(\alpha\bar\alpha+\beta\bar\beta)^2}.
\end{equation*}
The holomorphic extension of $E$ to $\GL(2,\C)$ is 
\begin{equation*}
 \psi_{0'}^\C=\frac{z_{11'}}{(z_{00'}z_{11'}-z_{01'}z_{10'})^2}\ \mathrm{and}\  \psi_{1'}^\C=\frac{-z_{10'}}{(z_{00'}z_{11'}-z_{01'}z_{10'})^2}.
\end{equation*}
\end{ex}

If $U\subset\Q^n$ is open and $x\in U$, we put $\delta(x,U^\complement):=\inf_{y\in U^\complement}\|x-y\|$ so that  $\delta(-,U^\complement)$ is continuous in $U$. If $U_\C\subset\C^{4n}$ is open, we similarly define $\delta^\C(z,U_\C^\complement):=\inf_{w\in U_\C^\complement}\|z-w\|, \ z\in U_\C$.

\begin{lemma}\label{lemma distance functions}
Let $U$ be an open subset of $\Q^n$. Then, with the notation set above, we have 
\begin{equation*}
 \delta^\C((x,y),\mh(U)^\complement)=\frac{1}{\sqrt2}\inf\limits_{q\in\Im\Q\cap\Sp(1)}\delta(x+yq,U^\complement),\ (x,y)\in\mh(U).
\end{equation*}
\end{lemma}
\begin{proof}
Fix $(x,y)\in\mh(U)$ and put $c:=\inf_{q\in\Im\Q\cap\Sp(1)}\delta(x+yq,U^\complement)$.
Then there are $x_\circ\in\partial U$ and $q\in\Im\Q\cap\Sp(1)$ such that $\|x+yq-x_\circ\|=c$. Put $w:=x_\circ-x-yq$ and consider $x'=x+\frac{1}{2}w$ and $y'=y-\frac{1}{2}wq$. Then
$$\|(x',y')-(x,y)\|_\C=\|\frac{1}{2}(w,-wq)\|_\C=\frac{\sqrt2}{2}\|w\|=\frac{c}{\sqrt2}$$
and 
$$x'+y'q=x+yq+w=x_\circ.$$
It follows that $(x',y')\not\in\mh(U)$ and thus $\delta^\C((x,y),\mh(U)^\complement)\le\frac{c}{\sqrt2}$.

On the other hand, choose $(x'',y'')\in M_{2n\times2}^\C$  with 
$$\|(x'',y'')-(x,y)\|_\C<\frac{c}{\sqrt2}.$$
We put $w':=x''-x$ and $w'':=y''-y$. If $q'\in\Im\Q\cap\Sp(1)$, then 
\begin{align*}
\|(x''+y''q')-(x+yq')\|&=\|w'+w''q'\|\\
&\le\sqrt{\|w'\|^2+\|w''\|^2+2(w',w'')}\\
&\le\sqrt{\|w'\|^2+\|w''\|^2+2\|w'\|\|w''\|}\\
&\le\sqrt{2(\|w'\|^2+\|w''\|^2)}<c.
\end{align*}
It follows that $(x'',y'')\in\mh(U)$ and thus $\delta^\C((x,y),\mh(U)^\complement)\ge\frac{c}{\sqrt2}$.
\end{proof}

Hence, we have the following

\begin{cor}\label{corollary domain of convergence}
If $\psi$ is a monogenic function in the ball $B_r:=\{x\in\Q^n:\ \|x\|<r\}$, then the Taylor series\footnote{The Taylor series of $\psi$ is in the variables $\alpha_\ell,\bar\alpha_\ell,\beta_\ell,\bar\beta_\ell,\ \ell=1,\dots,n$.} of $\psi$ centered at $0$  converges in $B_{\frac{r}{\sqrt2}}$ to $\psi$.
\end{cor}

Notice that Corollary \ref{corollary domain of convergence} is in accordance with \cite{Ha}.

\section{The Penrose transform for the $n$-Cauchy-Fueter operator}
In Section \ref{section cp1} we will review some well known material on sheaf cohomology groups of holomorphic line bundles over the Riemann sphere and provide some elementary proofs which will be used afterwards.

\subsection{Complex projective line}
\label{section sheaf coh groups over complex projective line}\label{section cp1}

We will use the standard homogeneous coordinates on $\CP^1$ and put $\mathfrak{X}_i:=\{[\pi_0:\pi_1]:\ \pi_i\ne0\},\ i=0,1$. Then there are biholomorphisms $\mathfrak{X}_0\ra\C,\ [\pi_0:\pi_1]\mapsto z:=\pi_1/\pi_0$ and $\mathfrak{X}_1\ra\C,\ [\pi_0:\pi_1]\mapsto w:=\pi_0/\pi_1$ with inverses $\C\ni z\mapsto[1:z]\in \mathfrak{X}_0$ and $\C\ni w\mapsto[w:1]\in \mathfrak{X}_1$, respectively.  We have that $\mathfrak{X}_0\cap \mathfrak{X}_1=\{[\pi_0:\pi_1]:\ \pi_0\pi_1\ne0\}=\{[1:z]: \ z\ne0\}=\{[w:1]:\ w\ne0\}\cong\C^\ast$ and that  $[1:z]=[w:1]$ if and only if $z=w^{-1}$. Hence, we can view $\mathfrak{X}_0$ and $\mathfrak{X}_1$ as $\C$ and we will do that without further comment.

We will denote by $Q_k,\ k\in\Z$ the holomorphic line bundle over $\CP^1$ with the transition function $z^{-k}$ in $\mathfrak{X}_0\cap \mathfrak{X}_1$. This means that  smooth functions $f_i:\mathfrak{X}_i\ra\C,\ i=0,1$  define a smooth section of $Q_k$ if 
\begin{equation}\label{clutching function over complex projective line}
 f_1(z^{-1})=z^{-k}f_0(z),\ \forall z\in\C^\ast.
\end{equation}
The section  is holomorphic if both functions are holomorphic. 

We will denote by $\Lambda^{(0,1)}\CP^1$ the vector bundle of $(0,1)$-forms over $\CP^1$, i.e. the fiber of this bundle over $x\in\CP^1$ is the vector space of all complex anti-linear maps $T_x\CP^1\ra\C$. The bundle $\Lambda^{(0,1)}\CP^1$ is trivialized by $d\bar z$ over $\mathfrak{X}_0$ and by  $d\bar w$ over $\mathfrak{X}_1$ with $d\bar w=-\bar z^{-2}d\bar z$ in $\mathfrak{X}_0\cap\mathfrak{X}_1$. It follows that a global smooth section of $\Lambda^{(0,1)}(Q_k):=\Lambda^{(0,1)}\CP^1\otimes\ Q^{k}$ is then given by a pair $(f_0\ d\bar z,f_1\ d\bar w)$ where $f_i:\mathfrak{X}_i\ra\C,\ i=0,1$ are smooth and
\begin{equation}\label{clutching function over complex projective line I}
f_1(z^{-1})=-z^{-k}\bar z^{2}f_0(z),\ \forall z\in\C^\ast.
\end{equation}
We denote by
\begin{equation*}
\cE(\CP^1,Q_k):=\Gamma(\CP^1,Q_k)\ \ \mathrm{and} \ \ \cE^{(0,1)}(\CP^1,Q_k):=\Gamma(\CP^1,\Lambda^{(0,1)}(Q_k)) 
\end{equation*}
the corresponding spaces of global sections. The  Dolbeault complex is 
\begin{displaymath}
\bpar:\cE(\CP^1,Q_k)\ra\cE^{(0,1)}(\CP^1,Q_k),\
\bpar(f_0,f_1)=\big(\partial_{\bar z}f_0\ d\bar z,\partial_{\bar w}f_1\ d\bar w\big).
 \end{displaymath}
We put  $H^0(\CP^1,Q_k):=\ker(\bpar)$ and $H^1(\CP^1,Q_k):=\coker(\bpar)$. By definition,  $H^0(\CP^1,Q_k)$ is  the space  of global holomorphic sections of $Q_k$.

$\CP^1$ can be also viewed as a 1-point compactification of $\C=\mathfrak{X}_0$ with the point $\infty=[0:1]$ at infinity, i.e.  $\mathfrak{X}_0$ is an open and dense subset of $\CP^1$ and thus each smooth section of a vector bundle over $\CP^1$ is uniquely determined by its restriction to $\mathfrak{X}_0$.

\begin{lemma}\label{lemma estimates}
 Let $k\in\Z$ and $f_0:U_0\ra\C$ be smooth. 
 
 (a) If $f_0$ extends to a global smooth section of $Q_k$, then
\begin{equation}\label{necessary con I}
\lim_{z\ra\infty} z^\ell f_0(z)\ \ \Big\{\begin{array}{cl}
                                         =0,&\ell< -k\\
                                         \in\C,&\ell=-k\\
                                        \end{array}
\end{equation}
(b) If $f_0\ d\bar z$ extends to a global smooth section of  $\Lambda^{(0,1)}(Q_k)$, 
then
\begin{equation}\label{necessary con II}
\lim_{z\ra\infty} z^\ell\bar z^n f_0(z)=0
\end{equation}
provided that $\ell+n< -k+2$
\end{lemma}
\begin{proof}
(a) If $(f_0,f_1)\in\cE(\CP^1,Q_k)$, it follows that  $z^\ell f_0(z)=z^{k+\ell} f_1(z^{-1})= w^{-k-\ell}f_1(w)$ where $w=z^{-1}\ne0$ and $\ell\in\Z$. Thus, (\ref{necessary con I}) is equal to $\lim_{ w\ra0} w^{-k-\ell}f_1( w)=f_1(0)\lim_{ w\ra0} w^{-k-\ell}$ and  the first claim follows. 

(b) If $(f_0\ d\bar z, f_1 d\bar w)\in\cE^{(0,1)}(\CP^1,Q_k)$, then we find that  $z^\ell \bar z^nf_0(z)=-w^{-k-\ell}\bar w^{-n+2}f_1( w)$ where $w=z^{-1}\ne0$ and $\ell\in\Z$. It follows that the limit in (\ref{necessary con II}) is equal to $f_1(0)\lim_{w\ra0}-w^{-k-\ell}\bar w^{-n+2}$. If $-k-\ell-n+2>0$, then it is zero.
\end{proof}

Assume that  $\omega:=(h_0(z)\ d\bar z,h_1(w)\ d\bar w)\in\cE^{(0,1)}(\CP^1,Q_k)$. By Lemma \ref{lemma estimates}, it follows that the integral
\begin{equation}\label{coefficients}
 a_{\ell}:=\frac{1}{2\pi i}\int_\C z^\ell h_0(z)\ d\bar z\wedge dz,\ 0\le\ell\le-k-2
\end{equation}
converges.

\begin{lemma}\label{lemma vanishing of cohomology groups}
If $k\le-2$, then the map
\begin{equation}\label{cech cohomology representative}
\omega\mapsto (a_{0},\dots,a_{-k-2})
\end{equation}
defined above descends to linear isomorphism $H^1(\CP^1,Q_k)\ra\C^{-k-1}$ while $H^1(\CP^1,Q_k)=\{0\}$ otherwise. 
\end{lemma}
\begin{proof}
 If $\omega$ is exact, say $\partial_{\bar z}f_0=h_0$  where $f_0: \mathfrak{X}_i\ra\C$ is smooth, then by  Stokes' theorem (see \cite[Section 1.3]{H}): 
$$\int_\C z^\ell h_0(z)\ d\bar z\wedge dz=\lim_{R\ra+\infty}\int_{S^1_R}z^\ell f_0dz=\lim_{R\ra+\infty}\int_{0}^{2\pi}f_0(Re^{it})iR^{\ell+1}e^{it(\ell+1)}dt$$
where $S^1_R=\{z\in\C:\ |z|=R\}$. Since $\ell+1<-k$, it follows by Lemma \ref{lemma estimates}(a)  that $\lim_{R\ra+\infty}R^{\ell+1}|f_0(Re^{it})|=0$ and thus the map (\ref{cech cohomology representative}) descends to cohomology.

It is easy to see that the map (\ref{cech cohomology representative}) is onto and thus, it remains to show that it is injective. Let us  assume that $a_0=\dots=a_{-k-2}=0$. By \cite[Theorem 1.4.4]{H}, there are functions $g_i:\mathfrak{X}_i\ra\C,\ i=0,1$ such that $\partial_{\bar z}g_0=h_0$ and $\partial_{\bar w}g_1=h_1$. Put $t:\mathfrak{X}_0\cap \mathfrak{X}_1\ra\C,\ t(z):=z^k g_1(z^{-1})$. Then $\partial_{\bar z}t=h_0$ in $\mathfrak{X}_0\cap \mathfrak{X}_1$, and thus $g_0-t$ is analytic in $\mathfrak{X}_0\cap \mathfrak{X}_1$, say $g_0-t=\sum_{i\in\Z}b_iz^i,\ b_i\in\C,\ z\ne0$. By the  definition of $t$ and the first part of the proof, it follows that $a_\ell=b_{-1-\ell},\ \ell=0,\dots,-k-2$.  Put $f_0:=g_0-\sum_{i\ge0}b_iz^i$ and $f_1:=g_1-\sum_{i\le k}b_{i}w^{k-i}$. Then $f_i:\mathfrak{X}_i\ra\C,\ i=0,1$ are smooth  with $\partial_{\bar z}f_0=h_0,\ \partial_{\bar w}f_1=h_1$ and $f_1(z^{-1})=z^{-k}f_0(z)$. We have proved that $f=(f_0,f_1)\in\cE(\CP^1,Q_k)$ and  $\omega=\bpar f$.
\end{proof}

\begin{ex}\label{example distinct cohomology class}
Notice that the map
\begin{align}\label{harmonic cohomology}
\left(
\begin{array}{c}
 a_0\\
 a_1\\
\end{array}
\right)
\mapsto2\bigg[\bigg(\frac{a_0 d\bar z+a_1\bar z d\bar z}{(1+z\bar z)^3},\frac{ -a_0\bar wd\bar w-a_1d\bar w}{(1+w\bar w)^3}\bigg)\bigg],
\end{align}
where $[\ \ ]$ denotes the corresponding cohomology class, is inverse  to the isomorphism $H^1(\CP^1,Q^{-3})\ra\C^{2}$ from Lemma \ref{lemma vanishing of cohomology groups}.
\end{ex}

\subsection{Double fibration diagram and correspondence}\label{section dfd}

Using  (\ref{lin isom Qn and C2n}), there is a well defined embedding $\iota:\QP^n\hookrightarrow V_{2}(\C^{2n+2})$ where $\QP^n$ is the quaternionic projective space in dimension $n$ and $V_{2}(\C^{2n+2})$ is the Grassmannian of complex 2-dimensional subspaces in $\C^{2n+2}$.  We will view  $\Q^n$ as the standard affine subset $\{[1:q_1:\dots:q_n]:\ q_\ell\in\Q,\ \ell=1,\dots,n\}$ of $\QP^n$. Now consider the map 
\begin{equation}\label{coordinates on affine of V2}
 (z_{AB'})\in M_{2n\times2}^\C\mapsto
 \left[
 \begin{matrix}
  1&0\\
  0&1\\
  z_{00'}&z_{01'}\\
  \vdots&\vdots\\
  z_{2n-1,0'}&z_{2n-1,1'}\\
 \end{matrix}
 \right]\in V_2(\C^{2n+2})
\end{equation}
where we denote by  square brackets the complex linear subspace spanned by the columns of the given $(2n+2)\times2$-matrix. The map identifies $M_{2n\times2}^\C$ with an open, dense and affine subset  of $V_2(\C^{2n+2})$ which we for brevity also denote by $M_{2n\times2}^\C$ and we will  view a $2n\times2$-matrix as the corresponding 2-plane in $\C^{2n+2}$  without further comment.
Altogether, there are inclusions
\begin{equation}\label{inclusions}
\begin{matrix}
\Q^n&\subset&\QP^n\\
\cap&&\cap\\
M_{2n\times2}^\C&\subset&V_2(\C^{2n+2}).\\
\end{matrix}
\end{equation}
where the embedding  $\Q^n\hookrightarrow M_{2n\times2}^\C$ is given in (\ref{embedding Qn into M2n2})

Consider the double fibration diagram 
\begin{equation}\label{double fibration diagram I}
\xymatrix{
&\Fl\ar[dl]^\eta\ar[dr]^\tau&\\
\CP^{2n+1}&&V_2(\C^{2n+2})}
\end{equation}
where $\Fl$ is the flag manifold of nested subspaces $(\ell,\Sigma)$  where $\ell\in\CP^n,\ \Sigma\in\cgr$ and $\ell\subset\Sigma$. The maps $\eta$ and $\tau$ are the obvious projections. The space on the left hand side is  called the \textit{twistor space} and the space in the middle upstairs is  called the \textit{correspondence space}. 

Let $U_c\subset\cgr$. We put $\hat U_c:=\tau^{-1}(U_c)$ and $\wt U_c:=\eta(\hat U_c)$ so  there is  another diagram
\begin{equation}\label{double fibration diagram for U}
\xymatrix{&\ar[dl]\hat U_c\ar[dr]&\\
\widetilde U_c&&U_c.}
\end{equation}
We for clarity put $\wt \Q^n:=\wt{\Q^n},\ \wt M_{2n\times2}^\C:=\wt{M_{2n\times2}^\C}$ and $\wt \Sigma:=\wt{\{\Sigma\}}$ where $\Sigma\in \cgr$. By definition, $\wt \Sigma$ is the set of all complex projective lines which are contained in $\Sigma$ and thus, $\wt\Sigma$ is biholomorphic to $\CP^1$. If $\Sigma$ is the 2-plane  on the right hand side of (\ref{coordinates on affine of V2}), then  $\wt\Sigma$ is the image of the embedding 
\begin{align}\label{embedding of the fiber of the correspondence}
\iota_\Sigma:\CP^1&\hookrightarrow\CP^{2n+1},\\
 [\pi_0:\pi_1]&\mapsto[\pi_0:\pi_1:\pi_0z_{00'}+\pi_1z_{01'}:\dots:\pi_0z_{2n-1,0'}+\pi_1z_{2n-1,1'}].\nonumber
\end{align}
It is easy to see that there is a biholomorphism
\begin{equation}\label{trivialization of tau}
\CP^1\times U_c\ra\hat U_c,\ ([\pi_0:\pi_1],\Sigma)\mapsto(\iota_\Sigma([\pi_0:\pi_1]),\Sigma)
\end{equation}
so that $\tau|_{\hat U_c}:\hat U_c\ra U_c$ corresponds to the projection onto the first factor.

It is clear that $\wt{M}_{2n\times2}^\C=W_0\cup W_1$ where
\begin{align}\label{coordinates on W0}
W_0&=\{[1:z_0:\dots:z_{2n}]:\ z_i\in\C,\ i=0,1,\dots,2n\}
\end{align}
and
\begin{align}\label{coordinates on W1}
W_1&=\{[w_0:1:w_1:\dots:w_{2n}]:\ w_i\in\C,\ i=0,1,\dots,2n\}
\end{align}
and that $W_0\cap W_1$ is the subset $\{z_0\ne0\}$ of $W_0$ and $\{w_0\ne0\}$ of $W_1$. The change of coordinates  is 
$$w_0=z_0^{-1},\ w_i=z_iz_0^{-1},\ i=1,\dots,2n.$$

\subsection{Correspondence over $\Q^n$}\label{section correspondence over Qn}
Let  $\mathfrak{X}_0,\mathfrak{X}_1$ be the open affine subsets of $\CP^1$ from Section \ref{section sheaf coh groups over complex projective line}
and $U$ be an open subset of $\Q^n$. By (\ref{inclusions}), we shall view  $U$  as an open subset of $V_2(\C^{n+2})$.  By (\ref{trivialization of tau}), there is a diffeomorphism $\CP^1\times U\ra\hat U$ such that $\CP^1\times U\cong\hat U\xrightarrow{\tau|_{\hat U}} U$ is  the canonical projection onto the second factor. 
We will view $\hat U_i:= \mathfrak{X}_i\times U=\C\times U,\ i=0,1$  as  open  subsets of $\hat U$.  It is easy to see that 
 $\widetilde U_i:=\eta(\hat U_i)=\widetilde U\cap W_i,\ i=0,1$ and that $\eta$ restricts to maps $\hat U_0\ra\widetilde U_0$ and $\hat U_1\ra\widetilde U_1$ which are given by
\begin{equation}\label{projection eta a}
(z,
\left(
\begin{matrix}
\alpha_1\\
\beta_1\\
\vdots\\
\alpha_n\\
\beta_n\\
\end{matrix}\right)
)\mapsto
\left[
\begin{matrix}
1\\
z\\
\alpha_1-z\bar\beta_1\\
\beta_1+z\bar\alpha_1\\
\vdots\\
\alpha_n-z\bar\beta_n\\
\beta_n+z\bar\alpha_n]
\end{matrix}
\right]
\ \ \mathrm{and} \ \
(w,\left(
\begin{matrix}
\alpha_1\\
\beta_1\\
\vdots\\
\alpha_n\\
\beta_n\\
\end{matrix}
\right))\mapsto
\left[
\begin{matrix}
w\\
1\\
w\alpha_1-\bar\beta_1\\
w\beta_1+\bar\alpha_1\\
\vdots\\
w\alpha_n-\bar\beta_n\\
w\beta_n+\bar\alpha_n\\
\end{matrix}
\right],
\end{equation}
respectively.
Using the notation from (\ref{coordinates on W0}) and (\ref{coordinates on W1}), we have
\begin{align*}
 z_0=z,\ z_{2i-1}=\alpha_i-z\bar\beta_i,\ z_{2i}=\beta_i+z\bar\alpha_i
\end{align*}
and
\begin{align*}
 w_0=w,\ w_{2i-1}=w\alpha_i-\bar\beta_i,\ w_{2i}=w\beta_i+\bar\alpha_i,
\end{align*}
 and conversely
\begin{align}\label{change of coordinates on complex projective space 0}
& z=z_0,\ \alpha_i=\frac{z_{2i-1}+z_0\bar z_{2i}}{1+z_0\bar z_0},\ \beta_i=\frac{z_{2i}-z_0\bar z_{2i-1}}{1+z_0\bar z_0}
\end{align}
and
\begin{align}\label{change of coordinates on complex projective space 1}
& w=w_0,\ \alpha_i=\frac{\bar w_{2i}+\bar w_0 w_{2i-1}}{1+w_0\bar w_0},\ \beta_i=\frac{-\bar w_{2i-1}+\bar w_0w_{2i}}{1+w_0\bar w_0}
\end{align}
where $i=1,\dots,n$. It is now clear  that the maps $\hat U_i\ra\wt U_i,\ i=0,1$ are diffeomorphisms and hence, we get the following important Lemma.

\begin{lemma}\label{lemma correspondence}
Let $U$ be an open subset of $\Q^n$. Then $\wt U$ is an open subset of $\CP^{2n+1}$ and
 $\eta:\hat U\ra\widetilde U$ is a diffeomorphism.
\end{lemma}

We can now give an equivalent characterization of the monogenic hull associated to $U$.

\begin{thm}\label{thm characterization of monogenic hull}
Let $U$ be an open subset of $\Q^n$. Then 
$$\mh(U)=\{\Sigma\in M_{2n\times2}^\C:\ \wt\Sigma\subset\widetilde U\}.$$
\end{thm}

\begin{proof}
First of all, it is  clear that $\wt\Q^n=\wt M_{2n\times2}^\C=W_0\cup W_1$ and that $\wt\Q^n=\bigcup_{\Sigma\in\Q^n}\wt{\Sigma}$ where the sum is disjoint. Hence, if we denote by $\Sigma_\ell\in\mathbb{HP}^n$ the unique quaternionic line which contains $\ell\in\CP^{2n+1}$, then we have
\begin{align*}
 \wt\Sigma\subset\widetilde U&\Leftrightarrow 
  \wt\Sigma\cap\widetilde{U}^\complement=\emptyset\\
  &\Leftrightarrow\wt\Sigma\cap\widetilde{U^\complement}=\emptyset\\
  &\Leftrightarrow
 \{\Sigma_\ell:\ \ell\in\wt\Sigma\}\cap U^\complement=\emptyset\\
 &\Leftrightarrow 
 \{\Sigma_\ell:\ \ell\in\wt\Sigma\}\subset U
\end{align*}
where we put $\wt U^\complement:=\wt \Q^n\setminus \wt U$.  Thus if $\Sigma=(x,y)$ where $x,y\in\C^{2n}=\Q^n$, then it is enough to show that
\begin{equation}\label{help}
\{\Sigma_\ell:\ \ell\in\wt\Sigma\}=\{x+yq:\ q\in\Im\Q\cap\Sp(1)\}.
\end{equation}
Recall (\ref{embedding Qn into M2n2}) that $\Sigma=(x,y)$ by definition means that $\Sigma=\mathrm{M}(x)+i\mathrm{M}(y)$ where $\mathrm{M}(x)=(x|\mathbb K(x))$, i.e. the first column of $\mathrm{M}(x)\in M_{2n\times2}^\C$ is $x$ and the second column is $\mathbb K(x)$. Observe that $i\mathrm{M}(y)=(iy |-\mathbb K(iy))$. It is a straighforward computation to verify that
$$\{\Sigma_\ell:\ \ell\in\wt \Sigma\}=\big\{
 x+i(\alpha\bar\alpha-\beta\bar\beta)y-2\beta\bar\alpha \mathbb K(iy)
\big|\ \alpha,\beta\in\C,\ \alpha\bar\alpha+\beta\bar\beta=1
\big\}.$$
By (\ref{lin isom Qn and C2n}),  $\mathbb K(iy)\in\C^{2n}$ corresponds to $yik=-yj\in\Q^n$ and thus, we see that $i(\alpha\bar\alpha-\beta\bar\beta)y-2\beta\bar\alpha\mathbb K(iy)\in\C^{2n}$ corresponds to
\begin{equation*}
y(\alpha\bar\alpha-\beta\bar\beta)i+2yj\bar\alpha\beta\in\Q^n.
\end{equation*}
Now it is easy to see that (\ref{help}) holds.
\end{proof}

\section{Dolbeault complex  over the twistor space}

\subsection{Filtration of the vector bundle of $(0,q)$-forms}

By Lemma \ref{lemma correspondence}, there are diffeomorphisms $\widetilde U\cong\hat U\cong\CP^1\times U$ and we for brevity denote the composition $\widetilde U\cong\hat U\xrightarrow{\tau} U$ also by $\tau$ as there is no risk of confusion.  Also recall Section \ref{section correspondence over Qn} that  $\wt U=\wt U_0\cup\wt U_1$ and that $\wt U_i=\mathfrak X_i\times U\cong \C\times U,\ i=0,1$.

The composition $\zeta: \wt U\cong\CP^1\times U\ra\CP^1$, where the second map is the canonical projection, is the restriction of the canonical projection $\CP^{2n+1}\ni[\pi_0:\pi_1:\dots:\pi_{2n+1}]\mapsto[\pi_0:\pi_1]$. We see that $\zeta$ is  holomorphic and thus  $L_k:=\zeta^\ast Q_k,\ k\in\Z$ is a holomorphic vector bundle over $\wt U$. By (\ref{clutching function over complex projective line}), it follows  that  a pair of smooth functions  $f_i:\wt U_i=\mathfrak{X}_i\times U\ra\C,\ i=0,1$ defines a global smooth section of $L_k$ if and only if
 \begin{equation}\label{clutching function over twistor space}
 f_1(z^{-1},x)=z^{-k}f_0(z,x),\ z\in\C^\ast,\ x\in U
\end{equation}
holds on $\wt U_0\cap\wt U_1=(\mathfrak{X}_0\cap \mathfrak{X}_1)\times U=\C^\ast\times U$.

\medskip

As $\wt U$ is an open subset of the complex manifold $\CP^{2n+1}$, then $T\wt U_\C:=T\wt U\otimes\C=T^{(1,0)}\oplus T^{(0,1)}$ where $T^{(0,1)}$ and $T^{(1,0)}$ is the $(+i)$-eigenspace and $(-i)$-eigenspace with respect to the canonical almost complex  structure on $T\wt U_\C$, respectively. We denote by $\Lambda^{(0,q)}$ the vector bundle over $\wt U$ whose fiber over $\ell$ consists of all skew-symmetric complex anti-linear maps  $\otimes^q T_\ell\wt U\ra\C$. 

Even though the projection  $\tau:\widetilde U\ra U$ is not holomorphic, it is a submersive surjection with fibers diffeomorphic to $\CP^1$. This induces a surjective vector bundle map $T\tau_\C:T\widetilde U_\C\ra TU_\C$. It follows that  $\ker(T\tau_\C)$ is a  subbundle of $T\wt U_\C$ of rank 1 and thus,  $K:=\ker(T\tau_\C)^\bot\cap\Lambda^{(0,1)}$  is a  subbundle of $\Lambda^{(0,1)}$ of co-dimension 1 which induces a short exact sequence
\begin{equation}\label{filtration of 01 forms}
 0\ra K\ra\Lambda^{(0,1)}\ra\Lambda_\tau^{(0,1)}\ra0
\end{equation}
and more generally,
\begin{equation}\label{filtration of 0q forms}
 0\ra \Lambda^{q+1}K\ra\Lambda^{(0,q+1)}\ra\Lambda^{(0,1)}_\tau\wedge\Lambda^{q}K\ra0,\ q\ge1.
\end{equation}
Let us now consider the bundles $K$ and $\Lambda^{(0,1)}_\tau$ over  $\wt U_i,\ i=0,1$.

Using (\ref{change of coordinates on complex projective space 0}), it is clear that the vector bundle $T^{0,1}$ is over $\widetilde U_0$ spanned by the vector fields
\begin{align}\label{anti-holomorphic vector fields over twistor space}
 \partial_{\bar z},\ X^{2i-1}_0:=z\partial_{\beta_i}-\partial_{\bar\alpha_i},\ X^{2i}_0:=z\partial_{\alpha_i}+\partial_{\bar\beta_i},\ i=1,\dots,n.
\end{align}
We denote by
$d\bar z, dX^i_0,\ i=1,\dots,2n$
the dual co-framing by $(0,1)$-forms which trivialize $\Lambda^{(0,1)}$ over $\wt U_0$. From (\ref{change of coordinates on complex projective space 1}), it follows that the bundle $T^{0,1}$ is over $\wt U_1$ spanned by the anti-holomorphic vector fields
\begin{align*}
\partial_{\bar w},\ X_1^{2i-1}:=\partial_{\beta_i}-w\partial_{\bar\alpha_i},\ X_1^{2i}:=\partial_{\alpha_i}+w\partial_{\bar\beta_i},\ i=1,\dots,n.
\end{align*}
We denote by 
$d\bar w,dX^i_1,\ i=1,\dots,2n$
the dual co-framing over $\wt U_1$ by $(0,1)$-forms. 

We see that $\ker(T\tau_\C)$ is  over $\widetilde U_0$ spanned by $\partial_z,\partial_{\bar z}$  and by $\partial_w,\partial_{\bar w}$ over $\widetilde U_1$. 
Hence,  $K$ is  over $\widetilde U_0$  spanned by $dX_0^1,\dots,dX_0^{2n}$ and by 
$dX^1_1,\dots, dX_1^{2n}$  over $\widetilde U_1$.
Notice that over $\wt U_0\cap\wt U_1$:
\begin{equation}\label{transition function for anti-hol fields}
 X_1^i=z^{-1}X_0^i\ \mathrm{and}\ dX_1^i=z dX_0^i,\ i=1,\dots,2n
\end{equation}
which implies

\begin{lemma}
$K$ is a holomorphic vector bundle   isomorphic to $\bigoplus_{1}^{2n}L_{1}$.
\end{lemma}
It also follows that the complex line bundle $\Lambda_\tau^{0,1}$ is   over $\widetilde U_0$ spanned by  $d\bar z+K$  and  by $d\bar w+K$ over $\widetilde U_1$. As there is no risk of confusion, we will for brevity write $d\bar z$ and $d\bar w$ instead of $d\bar z+K$ and $d\bar w+K$, respectively.

\subsection{Dolbeault complex}

Let us for brevity put $L:=L_{-3}$.
We will use the following conventions:
\begin{align*}
\Lambda^{(0,q)}(L)&:=\Lambda^{(0,q)}\otimes L,\ \Lambda^{(0,q)}_K(L):=\Lambda^{q}K\otimes L, \\ \Lambda^{(0,q)}_\tau(L)&:=\Lambda^{q-1}K\wedge\Lambda^{(0,1)}_\tau\otimes L,\ 
\cE^{(0,q)}_\ast(\wt U,L):=\Gamma(\wt U,\Lambda^{(0,q)}_\ast\otimes L)
\end{align*}
where $\ast\in\{\ ,K,\tau\}$.
The filtration (\ref{filtration of 0q forms}) turns the Dolbeault complex into a filtered complex:
\begin{equation}\label{filtered complex}
\xymatrix{&0\ar[d]&0\ar[d]\\
&\cE^{(0,1)}_K(\wt U, L)\ar[d]^\upsilon&\cE^{(0,2)}_K(\wt U, L)\ar[d]^\upsilon&\\
\cE^{(0,0)}(\wt U, L)\ar[r]^{\bpar}&\cE^{(0,1)}(\wt U, L)\ar[r]^{\bpar}\ar[d]^{\pi_\tau}&\cE^{(0,2)}(\wt U, L)\ar[r]^{\bpar}\ar[d]^{\pi_\tau}&\dots\\
&\cE^{(0,1)}_\tau(\wt U, L)\ar[d]&\cE^{(0,2)}_\tau(\wt U,\wt L)\ar[d]&\\
&0&0}
\end{equation}
We put:
\begin{displaymath}
 \bpar^0_o: \cE^{(0,0)}(\wt U, L)\xrightarrow{\bpar}\cE^{(0,1)}(\wt U, L)\xrightarrow{\pi_\tau}\cE^{(0,1)}_\tau(\wt U, L)\label{first diagonal dbar}
\end{displaymath}
and
\begin{displaymath}
\bpar^1_o:\cE_K^{(0,1)}(\wt U,L)\xrightarrow{\upsilon}\cE^{(0,1)}(\wt U, L)\xrightarrow{\bpar}\cE^{(0,2)}(\wt U, L)\xrightarrow{\pi_\tau}\cE^{(0,1)}_\tau(\wt U, L).
\end{displaymath}
so that  
\begin{displaymath}
 \bpar^0_o (f_0, f_1)=(\partial_{\bar z} f_0\ d\bar z, \partial_{\bar w} f_1\ d\bar w)
\end{displaymath}
and
\begin{displaymath}
 \bpar^1_o \sum_{i=1}^{2n}(f_0^i\ dX_0^i, f_1^i\ dX_1^i)=\sum_{i=1}^{2n}(\partial_{\bar z} f^i_0\ d\bar z\wedge dX_0^i, \partial_{\bar w} f_0^i\ d\bar w\wedge dX_1^i).\nonumber
\end{displaymath}
We will for brevity write $\bpar_o:=\bpar_o^i,\ i=0,1,\dots$

\section{Integral formula for the $n$-Cauchy-Fueter operator}

Let $U$ be an open subset of $\Q^n\subset\QP^n$.
Assume that $\Sigma\in \mh(U)\subset M_{2n\times2}^\C$ which  is by Theorem \ref{thm characterization of monogenic hull} equivalent to $\wt\Sigma\subset\wt U$. Recall (\ref{embedding of the fiber of the correspondence}) that $\wt\Sigma$ is equal to the image of the  embedding 
$\iota_\Sigma:\CP^1\hookrightarrow\wt U$. It is straightforward to verify that $\iota_\Sigma^\ast L_k\cong Q_k$ and  $\iota_\Sigma^\ast\Lambda^{(0,1)}_\tau(L_k)\cong\Lambda^{(0,1)}(Q_k)$. Hence, there is a well defined  composition of maps
\begin{equation}\label{integration along fibre}
\cE^{(0,1)}_\tau(\wt U,L_k)\ra \cE^{(0,1)}(\CP^1,Q_k)\ra H^1(\CP^1,Q_k)\ra\C^{-k-1},
\end{equation}
where the first map is the pullback associated to $\iota_\Sigma$, the second map is the canonical projection and the last map is the isomorphism from Lemma \ref{lemma vanishing of cohomology groups}. As in Section \ref{section sheaf coh groups over complex projective line}, we may assume that $-k-1>0$ so that (\ref{integration along fibre}) is non-zero.

Let us now assume that $\Sigma\in U$.
As (\ref{integration along fibre}) depends smoothly on $\Sigma$, it induces 
\begin{align}\label{integration}
\tau_\ast:\cE^{(0,1)}_\tau(\wt U,L_k)&\ra\cC^\infty(U,\C^{-k-1})\\
\omega=(h_0d\bar z,h_1d\bar w)&\mapsto(\psi_{0'},\dots,\psi_{(-k-1)'})\nonumber
\end{align}
where 
\begin{equation*}
 \psi_{A'}(x):=\frac{1}{2\pi i}\int_{\C}z^A h_0(z,x)\ d\bar z\wedge dz;\ A=0,\dots,-k-1;\ x\in U.
\end{equation*}
The integral converges as the functions 
$h_i:\wt U_i\ra\C,\ i=0,1$ satisfy the compatibility condition 
\begin{equation}\label{clutching function over twistor space I}
h_1(z^{-1},x)=-z^{-k}\bar z^{2}h_0(z,x), \ z\in\C^\ast,\ x\in U
\end{equation}
and so one can use the proof of Lemma \ref{lemma vanishing of cohomology groups}.

\begin{lemma}\label{lemma short exact sequences}
Put $L:=L_{-3}$. Then the sequences 
\begin{equation}\label{les I}
 0\ra\cE^{(0,0)}(\wt U, L)\xrightarrow{\bpar_o}\cE^{(0,1)}_\tau(\wt U, L)\xrightarrow{\tau_\ast}\cC^\infty(U,\C^2)\ra0
\end{equation}
and  
\begin{equation}\label{les II}
0\ra \cE_K^{(0,1)}(\wt U,L)\xrightarrow{\bpar_o} \cE^{(0,2)}_\tau(\wt U, L)\xrightarrow{\tau_\ast} \cC^\infty(U,\C^{2n\ast})\ra0
\end{equation}
are short exact.
\end{lemma}
\begin{proof}
We will prove only the exactness of (\ref{les I}) as, using  $K\otimes L_{-3}\cong\bigoplus_{i=1}^{2n}L_{-2}$, the proof of the exactness of (\ref{les II}) is analogous. As $H^0(\CP^1,Q_{-3})$ is zero, it follows  that also $\ker(\bpar_o)$ is zero. As $\tau_\ast$ is obviously  surjective (see also (\ref{splitting}) below), it remains to show that $\ker(\tau_\ast)=\im(\bpar_o)$.

Let $\omega$ be as in  (\ref{integration}).
Using Cauchy's integral formula and partition of unity underlying the open cover $\{\mathfrak X_i:\ i=0,1\}$ of $\CP^1$, it is easy to construct functions  $g_i:\wt U_i\cong \mathfrak{X}_i\times U\ra\C,\ i=0,1$ such that $\partial_{\bar z}g_0= h_0,\ \partial_{\bar w}g_1= h_1$. Moreover, if $\psi_{A'}=0,\ A=0,1$, then arguing as in the proof of Lemma \ref{lemma vanishing of cohomology groups}, one can find functions $t_i:\wt U_i\ra\C$ such that $\partial_{\bar z}t_0=0$ and $\partial_{\bar w}t_1=0$  and that $f_i:=g_i-t_i,\ i=0,1$  satisfy (\ref{clutching function over twistor space}). Then $f=(f_0,f_1)\in\cE^{(0,0)}(\wt U,L)$ and $\bpar_o f=\omega$.
\end{proof}

Using (\ref{harmonic cohomology}), it is clear that  
\begin{align}\label{splitting}
\cC^\infty(U,\C^2)\ra&\ \cE^{(0,1)}(\wt U,L),\\ 
(\psi_{A'})_{A=0,1}\mapsto&\ (\psi_{A'})_{A=0,1}^\sharp:= 2\bigg(\frac{ \psi_{0'}d\bar z+\psi_{1'}\bar z d\bar z}{(1+z\bar z)^3},\frac{-\psi_{0'}\bar wd\bar w-\psi_{1'}d\bar w}{(1+w\bar w)^3}\bigg)\nonumber
\end{align}
is a splitting of the map $\tau_\ast\circ\pi_\tau$. It is then straightforward to verify that there is  a commutative diagram 
\begin{equation}\label{commutative diagram}
\xymatrix{
\cE^{(0,1)}(\wt U, L_{-3})\ar[r]^{\bpar}\ar[d]^{\tau_\ast\circ\pi_\tau}&\cE^{(0,2)}(\wt U, L_{-3})\ar[d]^{\tau_\ast\circ\pi_\tau}\\
\cE(U,\C^2)\ar[r]^D&\cE(U,\C^{2n\ast})}
\end{equation}
where $D$ is the $n$-Cauchy-Fueter operator.

\begin{thm}\label{thm isom cohomology and monogenic functions}
Consider
\begin{equation}\label{composition in thm isom}
\tau_\ast\circ\pi_\tau:\cE^{(0,1)}(\wt U,L)\ra\cE^{(0,1)}_\tau(\wt U,L)\ra\cC^\infty(U,\C^2).
\end{equation}
If $\varOmega\in\cE^{(0,1)}(\wt U,L)$ is closed, then $\tau_\ast\circ\pi_\tau(\varOmega)\in\mathcal{R}(U)$ and the composition (\ref{composition in thm isom})  induces  isomorphism
\begin{equation}\label{isomorphism}
\PT: H^1(\widetilde U,L)\ra \mathcal{R}(U).
\end{equation}
\end{thm}

\begin{proof}
The first claim is an easy consequence of the commutativity of (\ref{commutative diagram}). By (\ref{les I}),  it follows that $\tau_\ast\circ\pi_\tau(\varOmega)=0$  provided that $\varOmega$ is exact. We see that (\ref{isomorphism}) is injective and it remains to show surjectivity. So assume that $(\psi_{A'})_{A=0,1}$ is monogenic  in $U$ and put $\varOmega:=(\psi_{A'})_{A=0,1}^\sharp$ for brevity. Then by the commutativity of (\ref{commutative diagram}) again, $\tau_\ast\circ\pi_\tau\circ \bpar(\varOmega)=0$ and by the exactness of (\ref{les II}), it follows that there is $\theta\in\cE^{(0,1)}_K(\wt U,L)$ such that $\bpar_o(\theta)=\pi_\tau\circ\bpar(\varOmega)$. Hence, $\bpar(\varOmega-\theta)\in\cE^{(0,2)}_K(\wt U,L)$. But arguing as in the proof Lemma \ref{lemma short exact sequences}, it is easy to see that $\cE^{(0,2)}_K(\wt U,L)\xrightarrow{\upsilon}\cE^{(0,2)}(\wt U,L)\xrightarrow{\bpar}\cE^{(0,3)}(\wt U,L)$ is injective. Hence, $\bpar(\varOmega-\theta)=0$ which completes the proof.
\end{proof}

Now  we are ready to proof the main result of this article.
\begin{proof}[Proof of Theorem \ref{main thm}]
As the only assumption on $\Sigma$ in (\ref{integration along fibre}) is that $\wt\Sigma\subset\wt U$, it follows that (\ref{integration along fibre}) induces a map
\begin{equation}\label{integration along fibres c}
\tau_\ast^\C:\cE^{(0,1)}_\tau(\wt U,L)\ra \cC^\infty(\mh (U),\C^{2}). 
\end{equation}
Explicitly, if $\omega$ is as in (\ref{integration}), then the value of $\tau_\ast^\C(\omega)$  at the point $\Sigma=(z_{AB'})_{A=0,1,\dots,2n-1}^{B=0,1}$  is 
$(\psi_{0'}^\C(\Sigma),\psi_{1'}^\C(\Sigma))$
where 
\begin{equation}\label{integral fmla for DC}
\psi_{A'}^\C(\Sigma):=\frac{1}{2\pi i}\int_\C z^A h_0(z,z_{00'}-zz_{01'},\dots,z_{2n-1,0'}-zz_{2n-1,1'})d\bar z\wedge dz.
\end{equation}
Arguing as in the proof of Theorem \ref{thm isom cohomology and monogenic functions}, it follows that $\tau_\ast^\C\circ\pi_\tau(\varOmega)=0$ whenever $\varOmega$ is exact. On the other hand, if $\varOmega$ is closed, then it is easy to see that $\tau_\ast^\C\circ\pi_\tau(\varOmega)$ is holomorphic and thus, $\tau_\ast^\C$ induces  a map
\begin{equation}\label{integral formula c}
\PT^\C: H^1(\wt U,L)\ra\cO(\mh(U),\C^{2}).
\end{equation}
Differentiating under the integral sign in (\ref{integral fmla for DC}), we see that $D^\C\tau_\ast^\C\circ\pi_\tau([\varOmega])=0$. Hence, the map (\ref{integral formula c}) takes values in $\mathcal{R}^\C(\mh(U))$ and we obtain a commutative diagram
\begin{equation*}
\xymatrix{H^1(\wt U,L)\ar[r]^{\PT^\C}\ar[dr]^{\PT}&\mathcal{R}^\C(\mh(U))\ar[d]\\
 &\mathcal{R}(U)}
\end{equation*}
where the vertical arrow is the restriction map.
Since the diagonal arrow is an isomorphism, it follows that the restriction map is surjective. As it is obviously injective, it is an isomorphism and thus, Theorem \ref{main thm} follow. 
\end{proof}

Notice that we have also shown
\begin{lemma}
 The map (\ref{integral formula c}) induces isomorphism
 \begin{equation*}
  H^1(\wt U,L_k)\ra\mathcal{R}^\C(\mh(U)).
 \end{equation*}
\end{lemma}

\def\bibname{Bibliography}

\textsc{Mathematical Institute of  Charles University, Sokolovsk\'a  49/83, Prague, The Czech Republic.}\\
\smallskip

\textit{E-mail address}: \texttt{salac$@$karlin.mff.cuni.cz}

\begin{thebibliography}{99}
\addcontentsline{toc}{chapter}{\bibname}
\bibitem{B} R.J. {\sc Baston}. \emph{Quaternionic complexes}. J. Geom. Phys., vol. 8 (1992), p. 29-52. 

\bibitem{BE} R.J.  {\sc Baston}, M. {\sc Eastwood}. \emph{The Penrose transform - its interaction   with representation theory}.   Oxford University Press, 1989. ISBN 0-19-853565-1.


\bibitem{BAS} J. {\sc Bure\v{s}}, A. {\sc Damiano}, I. {\sc Sabadini}. \emph{Explicit resolutions for the complex of several Fueter operators}. J. Geom. Phys. vol. 57 (2007), no.3, p. 765-775.

\bibitem{BS} J. {\sc Bure\v{s}}, V. {\sc Sou\v{c}ek}. \emph{Complexes of invariant dif and only iferential operators in several quaternionic variables}. Compl. Var. and Ell. Eq., vol. 51 (2006), no. 5-6, p. 463-487.
 
\bibitem{CS}
A.  {\sc \v Cap,}  J. {\sc Slov\'ak}. \emph{Parabolic Geometries I, Background and General Theory}.
  Am. Math. Soc., Providence, 2009. ISBN 978-0-8218-2681-2.

\bibitem{CSSS}
   F. {\sc Colombo}, I. {\sc Sabadini}, F. {\sc Sommen}, D. C. {\sc Struppa}.
  \emph{ Analysis of Dirac Systems and Computational Algebra}. Birkhauser, Boston, 2004. ISBN 0-8176-4255-2.

\bibitem{CSS}
  F. {\sc Colombo,}, V. {\sc Sou\v cek}, D. C. {\sc Struppa}.   \emph{Invariant resolutions for several Fueter operators}. J.  Geom. Phys., vol. 56 (2006), no. 7, p. 1175-1191. 

\bibitem{E} M. {\sc Eastwood}. \emph{The twistor construction and Penrose transform in split signature}. Asian J. Math. vol. 11 (2007), no. 1, p. 103-112.  

\bibitem{EX} M. {\sc Eastwood}, F. {\sc Xu}. \emph{The harmonic hull and twistor theory}. Lie Groups: Structure, Actions, and Representations
vol. 306 of the series Progress in Mathematics (2013), p. 59-80.  ISBN 978-1-4614-7192-9.

\bibitem{Ha} W. K. Hayman. \emph{Power series expansions for harmonic functions}. Bull. Lond. Math. Soc., vol. 2 (1970), is. 2, p. 152-158.

\bibitem{H} L. {H\"ormander}. \emph{An introduction to complex analysis in several variables}. D. Van Norstrand Company, Inc., Princeton, New Jersey, 1966.


\bibitem{KW} Q. {\sc Kang},  W. {\sc Wang}. \emph{On Penrose integral formula and series expansion of
-regular functions on the quaternionic space $\Q^n$}. J. Geom. Phys., vol. 64 (2013), p. 192-208.

\bibitem{Pe} D. {\sc Pertici}. Funzioni regolari di pi\'u variabili quaternioniche. Ann. Mat. Puera e Appl. Serie IV, CLI (1988), p. 39-65.

\bibitem{SSSL}
I.  {\sc Sabadini}, F. {\sc Sommen}, D. C. {\sc Struppa}, P. van {\sc Lancker}.
  \emph{Complexes of Dirac operators in Clifford algebras}.
Mathematische Zeitschrift. 2002, vol. 239, num. 2, p. 293-320. ISSN: 0025-5874.

\bibitem{TS} T. {\sc Sala\v c}. Domains of monogenicity and twistor theory. to appear.

\bibitem{TSI} T. {\sc Sala\v c}. Domains of monogenicity and pseudoconvexity. to appear.

\bibitem{Wa} W. {\sc  Wang}. \emph{On non-homogeneous Cauchy-Fueter equations and Hartogs’ phenomenon in several quaternionic variables}. J. Geom. Phys.
vol 58  (2008), p. 1203-1210

\bibitem{WaI} W. {\sc Wang}. \emph{The $k$-Cauchy-Fueter complex, Penrose transformation and Hartogs phenomenon for quaternionic $k$-regular functions}. J. Geom.  Phys. vol 60  (2010), p. 513-530

\bibitem{WW} R.S. {\sc Ward}, R.O. {\sc Wells} Jr. \emph{Twistor geometry and Field theory}. Cambridge University Press 1990, ISBN 0-521-42268-X
\end{thebibliography}
\end{document}